\documentclass[12pt]{amsart}

\def\adj{\operatorname{adj}}

\usepackage{color}
\definecolor{red}{rgb}{1.00,0.00,0.00}

\newtheorem{theorem}{Theorem}
\newtheorem{lemma}[theorem]{Lemma}
\newtheorem{corollary}[theorem]{Corollary}

\newtheorem{example}[theorem]{Example}
\newtheorem{remark}[theorem]{Remark}
\newtheorem{definition}[theorem]{Definition}

\title{A note on Gorenstein monomial curves}

\author{Philippe Gimenez}
 \address{Instituto de Matem\'aticas de la Universidad de Valladolid (IMUVA)\\
 Departamento de \'Algebra, An\'alisis Matem\'atico, Geometr\'{\i}a y Topolog\'{\i}a\\
 Facultad de Ciencias, Universidad de Valladolid, 47011 Valladolid, Spain.}
 \email{pgimenez@agt.uva.es}
 \thanks{The first author was partially supported by {\it Ministerio de Ciencia e Innovaci\'on} (Spain),
MTM2010-20279-C02-02.}
\thanks{The second author acknowldeges with pleasure the support and hospitality of University of Valladolid and the University of Missouri Research Council for their support. }
\author{Hema Srinivasan}
 \address{Mathematics Department, University of
Missouri, Columbia, MO 65211, USA.}
 \email{SrinivasanH@math.missouri.edu}

\begin{document}
\maketitle

\begin{abstract}
Let $k$ be an arbitrary field.
In this note, we show that if a sequence of relatively prime positive integers ${\bf a}=(a_1,a_2,a_3,a_4)$ defines a Gorenstein non complete intersection monomial curve ${\mathcal C}({\bf a})$ in
${\mathbb A}_k^4$, then there exist two vectors ${\bf u}$ and ${\bf v}$ such that
${\mathcal C}({\bf a}+t{\bf u})$ and ${\mathcal C}({\bf a}+t{\bf v})$ are also Gorenstein non complete intersection affine monomial curves
for almost all $t\geq 0$.
\\
{\sc Keywords:} affine monomial curve, numerical semigroup, Gorenstein curve.
\\
{\sc Mathematical subject classification:} 13C40, 14H45, 13D02, 20M25.
\end{abstract}

\bigskip

Let ${\bf a} = (a_1, \ldots a_n)$ be a sequence of positive integers and $k$ be an arbitrary field.
If $\phi:  k[x_1, \dots, x_n ]\to k[t]$ is the ring homomorphism defined by
$\phi(x_i) = t^{a_i}$, then $I({\bf a}):= \ker \phi$ is a prime ideal of height $n-1$ in $R:=k[x_1, \ldots, x_n]$
which is a weighted homogeneous binomial ideal with the weighting $\deg x_i:= a_i$ on $R$. It is the
defining ideal of the affine monomial curve ${\mathcal C}({\bf a})\subset {\mathbb A}_k^n$ parametrically defined by ${\bf a}$
whose coordinate ring is $S({\bf a }):= {\rm Im }\phi = k[t^{a_1}, \ldots, t^{a_n}]\simeq R/I({\bf a})$. As $S({\bf a})$ is isomorphic to $S(d{\bf a})$ for all integer $d\geq 1$,
we will assume without loss of generality that $a_1, \ldots, a_n$ are relatively prime. Observe that $S({\bf a})$ is also the semigroup ring of
the numerical semigroup $\langle a_1, \ldots a_n\rangle\subset{\mathbb N}$ generated by $a_1, \ldots ,a_n$.

\medskip

As observed in \cite{De} where Delorme characterizes sequences ${\bf a}$ such that $S({\bf a })$ is a complete intersection,
this fact does not depend on the field $k$ by \cite[Corollary~1.13]{He}.
On the other hand, it is well-known that $S({\bf a })$ is Gorenstein if and only if the numerical semigroup $\langle a_1, \ldots a_n\rangle\subset{\mathbb N}$
is symmetric, which does not depend either on the field $k$. We will thus
say that ${\bf a}$ is a complete intersection (respectively Gorenstein) if the semigroup ring $S({\bf a})$ is a complete intersection (respectively Gorenstein).
In \cite{JS}, it is shown that if ${\bf a}$ is a complete intersection
with $a_1>>0$, then ${\bf a} +t(a_n-a_1)(1,\dots ,1)$ is also a complete intersection for all $t$. In this note, we will use the criterion for Gorenstein
monomial curves in ${\mathbb A}_k^4$ due to Bresinsky in \cite{Br} to construct a class of Gorenstein
monomial curves in ${\mathbb A}_k^4$.

\medskip

First observe that for each $i$, $1\le i\le n$, there exists a multiple of $a_i$ that belongs to the numerical semigroup generated by the rest of the elements in the sequence and denote by $r_i>0$ the smallest positive integer such that $r_ia_i \in \langle a_1,\ldots,a_{i-1},a_{i+1},\ldots,a_n\rangle$. So we have that
\begin{equation}\label{principalrelations}
\forall i,\ 1\leq i\leq n,\ r_ia_i = \sum _{j\neq i} r_{ij}a_j,\ r_{ij} \ge 0,\ r_i >0\,.
\end{equation}

\begin{definition}{\rm
The $n\times n$ matrix $D(\bf {a}):=(r_{ij})$ where $r_{ii}:= -r_i$ is called a
{\it principal matrix} associated to $\bf {a}$.
}\end{definition}

\begin{lemma}
$D(\bf {a})$ has rank $n-1$.
\end{lemma}

\begin{proof}
%
Since the system $D({\bf a})X=0$ has solution $X={\bf a}^T$ by (\ref{principalrelations}), $D(\bf {a})$ has rank $\le n-1$.    If ${\bf b}^T$ is another solution, then the ideal defining the curve $I({\bf b})$ contains $$f_i = x_i^{r_i} -\prod_{j\neq i}x_j^{r_{ij}}$$
for all $1\le i\le n$.  But $I({\bf a})= \sqrt{(f_1, \ldots, f_n)}$ thus $I({\bf a})\subseteq I({\bf b})$ because $I({\bf b})$ is prime.
Since both $I({\bf a})$ and $I({\bf b})$ are primes of the same height $n-1$, they must be equal: $I({\bf a})= I({\bf b})$.
Since $x_i^{a_j}-x_j^{a_i}\in I({\bf a})$ for all $i$ and $j$, one has that
$a_ib_j = a_jb_i$ for all $i$ and $j$ and hence the $2\times n$ matrix whose rows are ${\bf a}$ and ${\bf b}$ is of rank 1, i.e.,
$a_i = cb_i$ for some $c$.  So, rank of $D({\bf a}) = n-1$.
\end{proof}

\begin{remark}{\rm
Observe that $D({\bf a})$ is not uniquely defined. Although the diagonal entries $-r_i$ are uniquely determined, there is not a unique choice for $r_{ij}$ in general.
We have the ``map" $D: {\mathbb N}^{[n]} \to T_n$ from the set ${\mathbb N}^{[n]}$ of sequences of $n$ relatively prime positive integers to the subset $T_n$ of $n\times n$ matrices of rank $n-1$ with negative integers on the diagonal and non negative integers outside the diagonal.
Note that we can recover $\bf {a}$ from $D(\bf {a})$ by factoring out the greatest common divisor of the $n$ maximal minors
of the $n-1\times n-1$ submatrix of $D(\bf {a})$ obtained by removing the first row. In other words, call $D^{-1}:T_n \to N^{[n]}$
the operation that, for $M\in T_n$, takes the first column of $\adj (M)$ and then factors out the g.c.d. to get an element in $N^{[n]}$. Then,
$D^{-1}(D({\bf a}))={\bf a}$ for all ${\bf a}\in N^{[n]}$.
Now given a matrix $M\in T_n$, $D(D^{-1}(M)) \neq M$ in general as the following example shows: if
$M=\left[
\begin{matrix}
-4&0&1&1\\
1&-5&4&0\\
0&4&-5&1\\
3&1&0&-2\\
\end{matrix}
\right]$ then $D^{-1}(M)=(7,11,12,16)$ and $D(D^{-1}(M) \neq M$ (it is easy to check for example that $r_2=3<5$).
}\end{remark}

\medskip

We now focus on the case of Gorenstein monomial curves in ${\mathbb A}_k^4$ so assume that $n=4$.
If ${\bf a}$ is Gorenstein but is not a complete intersection,
by the characterization in \cite[Theorems~3 and 5]{Br}, there is a principal matrix $D({\bf a})$ that has the following form:
\begin{equation}\label{embdim4GorMat}
\left[
\begin{matrix}
-c_1&0& d_{13} &d_{14}\\
d_{21}&-c_2&0&d_{24}\\
d_{31}&d_{32}& -c_3&0\\
0&d_{42}&d_{43}&-c_4\\
\end{matrix}
\right]
\end{equation}
with $c_i\ge 2$ and $d_{ij}>0$ for all $1\le i,j\le 4$, the columns summing to zero and all the columns of the adjoint being relatively prime.
The first column of the adjoint of this matrix is $-{\bf a}^T$ and Bresinsky's characterization also says that
the first column (after removing the signs) of the adjoint of a principal matrix $D({\bf a})$ of this form defines a Gorenstein curve
provided the entries of this column are relatively prime.

\medskip
The following is a slight strengthening of this criterion.

\begin{theorem}\label{criterion}
Let $A$ be a $4\times 4$ matrix of the form
$$
A=
\left[
\begin{matrix}
-c_1&0& d_{13} &d_{14}\\
d_{21}&-c_2&0&d_{24}\\
d_{31}&d_{32}& -c_3&0\\
0&d_{42}&d_{43}&-c_4\\
\end{matrix}
\right]
$$
with $c_i\ge 2$ and $d_{ij}>0$ for all $1\le i,j\le 4$, and all the columns summing to zero.  Then the  first column of the adjoint of $A$
{\rm(}after removing the signs{\rm)}
defines a monomial curve provided these entries are relatively prime.
\end{theorem}

\begin{proof}
Consider such a matrix $A$ and let $a_1, a_2, a_3, a_4$ be the entries in the first column of the adjoint of $A$ (after
removing the signs). Since we are assuming that they are relatively prime, there exist integers $\lambda_1,\ldots,\lambda_4$ such that
$\lambda_1 a_1+\cdots+\lambda_4 a_4=1$.

\smallskip

It suffices to show that the four relations in the rows of $A$ are principal relations.
We will show this for the first row and the other rows are similar.
Suppose that $b_{11} a_1 = b_{12}a_2+b_{13}a_3+b_{14}a_4$ is a relation with $b_{11}\geq 2$ and $b_{12}, b_{13}, b_{14}\geq 0$
and let's show that $b_{11}\geq c_1$.

\smallskip

Since the system
$
\left[
\begin{matrix}
-b_{11}&b_{12}& b_{13} &b_{14}\\
d_{21}&-c_2&0&d_{24}\\
d_{31}&d_{32}& -c_3&0\\
0&d_{42}&d_{43}&-c_4\\
\end{matrix}
\right]Y = 0$ has a nontrivial solution, namely $Y = (a_1, a_2, a_3, a_4)^T$,
we see that it has determinant zero.
So there exist $x_i$ such that
\begin{equation}\label{detzero}
(1,x_2, x_3, x_4)\left[
\begin{matrix}
-b_{11}&b_{12}& b_{13} &b_{14}\\
d_{21}&-c_2&0&d_{24}\\
d_{31}&d_{32}& -c_3&0\\
0&d_{42}&d_{43}&-c_4\\
\end{matrix}
\right]= 0\,.
\end{equation}

Consider the matrix $T_4 =\left[
\begin{matrix}
-b_{11}&b_{12}& b_{13} &b_{14}\\
d_{21}&-c_2&0&d_{24}\\
d_{31}&d_{32}& -c_3&0\\
\lambda_1&\lambda_2 &\lambda_3 &\lambda_4\\
\end{matrix}
\right]$.
If the determinant of $T_4$ is $-t_4$, then the last column of its adjoint is $-t_4 (a_1, a_2, a_3, a_4)^T$.
This is because $T_4 (a_1, a_2, a_3, a_4)^T = (0,0,0,1)^T$.
Hence, looking at the element in the last row and last column of the adjoint of $T_4$, one gets using (\ref{detzero}) that
$$- t_4 a_4= \left| \begin{matrix}
-b_{11}&b_{12}& b_{13}  \\
d_{21}&-c_2&0 \\
d_{31}&d_{32}& -c_3 \\
\end{matrix}
\right| =  \left| \begin{matrix}
0&-x_4d_{42}&-x_4 d_{43}  \\
d_{21}&-c_2&0 \\
d_{31}&d_{32}& -c_3 \\
\end{matrix}
\right| = -x_4a_4\,.$$
Hence $t_4 = x_4$, and since $t_4$ is an integer, so is $x_4$.
Now, looking at the element in the last column and first row of the adjoint of $T_4$, one has
$$t_4a_1 =    \left| \begin{matrix}
 b_{12}&b_{13}& b_{14}  \\
 -c_2&0&d_{24} \\
d_{32}& -c_3&0 \\
\end{matrix}
\right| = b_{12}c_3d_{24}+ b_{13}d_{32}d_{24}+b_{14}c_2c_3>0\,.$$
So, $x_4 = t_4$ is now a positive integer.

\smallskip

Consider the matrix $T_2 = \left[
\begin{matrix}
-b_{11}&b_{12}& b_{13} &b_{14}\\
d_{31}&d_{32}&-c_3&0 \\
0&d_{42}&d_{43}& -c_4\\
\lambda_1&\lambda_2 &\lambda_3 &\lambda_4\\
\end{matrix}
\right]$
which determinant is denoted by $-t_2$. By similar calculations, we see that $x_2=t_2$ is an integer and,
focusing on the element in the last column and third row of the adjoint of $T_2$, one gets that
$$
t_2a_3 =    \left| \begin{matrix}
 -b_{11}&b_{12}& b_{14}  \\
 d_{31}&d_{32}&0 \\
 0& d_{42}&-c_4 \\
\end{matrix}
\right| = b_{11}c_4d_{32}+ b_{12}d_{31}c_{4}+b_{14}d_{31}d_{42}>0
$$
so $x_2=t_2$ is also a positive integer.

\smallskip

Similarly, using the matrix $T_3=  \left[
\begin{matrix}
-b_{11}&b_{12}& b_{13} &b_{14}\\
d_{21}&-c_2&0&d_{24}\\
0&d_{42}&d_{43}& -c_4\\
\lambda_1&\lambda_2 &\lambda_3 &\lambda_4\\
\end{matrix}
\right]$ of determinant $-t_3$, one gets that
$x_3 = -t_3$ and hence $x_3$ is an integer.
However, by calculating the entry in the last column and second row of the adjoint of $T_3$, one gets that
$$(-t_3)a_2 =  \left| \begin{matrix}
- b_{11}&b_{13}& b_{14}  \\
d_{21}&0&d_{24} \\
0& d_{43}&-c_4 \\
\end{matrix}
\right| = b_{11}d_{43}d_{24}+b_{13}d_{21}c_4+b_{14}d_{21}d_{43}>0,$$
and hence, $x_3$ is again a positive integer.

\smallskip

So, $b_{11} = x_2d_{21}+x_3d_{31}\ge d_{21}+d_{31} = c_1$ as desired.

\smallskip

Since we can make any of the $c_i$'s the first row, by rearranging the $a_i$'s suitably, this proves that all of the rows are principal relations and this is a principal matrix.  Hence ${\bf a}=(a_1, a_2, a_3, a_4)$ is Gorenstein by Bresinsky's criterion.
\end{proof}

Denote now the sequence of positive integers by ${\bf a}=(a,a+x,a+y,a+z)$ for some $x,y,z>0$. In other words,
we assume that the first integer in the sequence is the smallest but after that we do not assume any ascending order.
Recall that we have assumed, without loss of generality, that ${\rm gcd}(a,x,y,z)=1$.
The following result gives two families of Gorenstein monomial curves in ${\mathbb A}_k^4$ by translation from a given Gorenstein curve.

\begin{theorem}\label{main}
Given any Gorenstein non complete intersection monomial curve ${\mathcal C}({\bf a})$ in ${\mathbb A}_k^4$, there exist two vectors ${\bf u}$ and ${\bf v}$ in ${\mathbb N}^4$ such that for all $t\ge 0$, ${\mathcal C}({\bf a} +t{\bf u})$ and ${\mathcal C}({\bf a}+t{\bf v})$ are also Gorenstein non complete intersection monomial curves whenever the entries of the corresponding sequence {\rm(}${\bf a} +t{\bf u}$ for the first family, ${\bf a} +t{\bf v}$ for the second{\rm)} are
relatively prime.
\end{theorem}

\begin{proof}
Let $D({\bf a})$ be the principal matrix of ${\bf a}$ given in (\ref{embdim4GorMat}).
Then let ${\bf u}$ be the vector of $3 \times 3 $  minors of the $3\times 4$ matrix
with
$U = \left[
\begin {matrix}
d_{21}&-c_2&0&d_{24}\\
1&0 & -1&0\\
0&d_{42}&d_{43}&-c_4\\
\end{matrix}\right]
$
so that $U {\bf u}^T = (0,0,0)^T$.
Focusing on the second row, we see that $u_1 = u_3$.  Similarly, let ${\bf v}$ be the third adjugate of the $3\times 4$ matrix
$V = \left[
\begin{matrix}
-c_1&0& d_{13} &d_{14}\\
0&-1&0&1\\
d_{31}&d_{32}& -c_3&0\\
\end{matrix}
\right]
$ so that $V {\bf v}^T = (0,0,0)^T$.

\smallskip

We will check now that, as long as their entries are relatively primes, the sequences
${\bf a} +t{\bf u}$ and ${\bf a}+t{\bf v}$ respectively have principal matrices
$$
A_t = \left[
\begin{matrix}
-c_1-t&0& d_{13} +t&d_{14}\\
d_{21}&-c_2&0&d_{24}\\
d_{31}+t&d_{32}& -c_3-t&0\\
0&d_{42}&d_{43}&-c_4\\
\end{matrix}
\right]
\quad\hbox{and}\quad
B_t = \left[
\begin{matrix}
-c_1&0& d_{13} &d_{14}\\
d_{21}&-c_2-t&0&d_{24}+t\\
d_{31}&d_{32}& -c_3&0\\
0&d_{42}+t&d_{43}&-c_4-t\\
\end{matrix}
\right]$$
and hence define Gorenstein curves.
It is a straightforward calculation to check that
the rows of these matrices are the relations of  ${\bf a} +t{\bf u}$ and ${\bf a}+t{\bf v}$ respectively, i.e.,
$A_t\times ({\bf a} +t{\bf u})^T=(0,0,0,0)^T$ and $B_t\times ({\bf a} +t{\bf v})^T=(0,0,0,0)^T$ and it suffices to check it for $t=1$.
Consider the vector $A_1 \times ({\bf a} + {\bf u})^T$.
If we add the first row of $D({\bf a})$ to $U$ to make a square matrix $U'$ then the determinant of $U'$, expanding by its third row, is $a-(a+y)=-y$.
On the other hand, the adjoint of $U'$ has ${\bf u}$ as the first column
and ${\bf a}$ as the third column,
thereby the first row of $U'$ multiplied by ${\bf u}$ equals $-y$
and the third row of $U'$ multiplied by ${\bf a}$ also equals $-y$.
Since the first row of $A_1$ is the first row of $U'$ minus the third row of $U'$, we see that the first entry
of the vector $A_1 \times ({\bf a} + {\bf u})^T$ is zero, and a similar argument works to show that the third entry
of this vector is also zero.
Moreover, the second row of $A_1$ coincides with the second row of $D({\bf a})$ so one gets 0 multiplying by ${\bf a}$, and since it is also the
first row of $U$, one also gets 0 multiplying by ${\bf u}$ and hence the second entry in the vector $A_1\times ({\bf a}+{\bf u})^T$ is zero. The same argument
works for the fourth entry and the proof for ${\bf v}$ is similar.

\smallskip

Since the differences between the matrix $D({\bf a})$ and $A_t$ are all in the first and third rows, and $U$ comes from the second and fourth row, we have shown that if $A_1$ is a relation matrix for ${\bf a}+{\bf u}$ , $A_{0} = D({\bf a})$ is a relation matrix for ${\bf a}$.
Moreover, note that the changes in $D({\bf a})$ to get $A_t$ or $B_t$ did not alter the column sums and hence the columns still add up to zero.

\smallskip

Since $A_1$ has the form, with zeros above the diagonal, with zero in the last column first row, all the non diagonal entries non negative and has all the columns sum to zero, it is principal provided  ${\bf a}+{\bf u}$ is relatively prime by Theorem \ref{criterion}.
Thus, if the cofactors are relatively prime, they form Gorenstein curves.
\end{proof}

Note that since we have assumed that the first entry in ${\bf a}$ is the smallest, the first principal relation can not be homogeneous (w.r.t. the usual grading
on $R$). Let us see what happens when 2 of the other 3 principal relations are homogeneous.

\begin {corollary}\label{corMultConj}
Let ${\bf a} = (a, a+x, a+y, a+z)$ be Gorenstein and not a complete intersection.
Suppose that both the second and the fourth rows of the matrix $D({\bf a})$ in (\ref{embdim4GorMat}) have their entries summing to zero.
Then, $x<z<y$ and ${\bf a} + t\alpha y(1,1,1,1)$ is Gorenstein for all $t\ge 0$, where  $\alpha$ is a positive integer determined by $(a,a+x,a+y,a+z)$.
\end{corollary}

\begin {proof}
Since the entries of the second row sum to 0, i.e., $c_2=d_{21}+d_{24}$, we get from $c_2(a+x)=d_{21}a+d_{24}(a+z)$ that
\begin{equation}\label{row2sumto0}
c_2x=d_{24}z,
\end{equation}
i.e., $(d_{21}+d_{24})x=d_{24}z$, and hence $x<z$.
Similarly, the sum of the entries of the fourth row being 0 implies that
\begin{equation}\label{row4sumto0}
d_{42}x+d_{43}y=c_4z,
\end{equation}
i.e., $d_{42}x+d_{43}y=(d_{42}+d_{43})z$, and hence
$d_{43}(y-z) = d_{42}(z-x)$ and one has that $y>z$.

Moreover, the hypothesis in the corollary, we compute that the vector ${\bf u}$ in the proof of the theorem~\ref{main} is ${\bf u} = b(1,1,1,1)$, where $b =  d_{21}c_4+d_{24}d_{43}$.

Now set $d:={\rm gcd}(x,z)$. Simplifying (\ref{row2sumto0}) by $d$, one gets that $c_2x/d=d_{24}z/d$ with ${\rm gcd}(x/d,z/d)=1$ so
$c_2=qz/d$ and $d_{24}=qx/d$ for some integer $q$.
Now simplifying (\ref{row4sumto0}) by $d$ also, one gets that $d$ divides $d_{43}y$ and $d_{43}y/d=c_4z/d-d_{42}x/d$
and hence $qd_{43}y/d=c_2c_4-d_{24}d_{42}=c_4d_{21}+d_{24}(c_4-d_{42})=c_4d_{21}+d_{24}d_{43}=b$.
Multiply $b$ by the smallest strictly positive integer $\beta$ so that $\beta b=  \alpha y$ for some integer $\alpha$. Note that $\beta\leq d$ and if $t=1$, then $\alpha=qd_{43}$.
Then,
${\bf a} + t\alpha y(1, 1,1,1)$ are always Gorenstein for if they have a common factor, then it will be a common factor of
$a+ t\alpha y$ and $a+y+t\alpha y$  which necessarily is common factor of both $a$ and $y$.
But then, it will be a common factor of $x$ and $z$, i.e., of all entries in ${\bf a}$ which are relatively prime.
\end{proof}

\begin{remark}{\rm
The previous result shows that for the Gorenstein curves satisfying the hypothesis in Corollary~\ref{corMultConj}, the periodicity conjecture
holds with period $\alpha y$.
}\end{remark}

\begin{example}{\rm
The sequence ${\bf a}=(11,17,25,19)$ is Gorenstein and not a complete intersection and its principal matrix,
$\displaystyle{
D({\bf a})=
\left[
\begin{matrix}
-4&0&1&1\\
1&-4&0&3\\
3&1&-2&0\\
0&3&1&-4\\
\end{matrix}
\right]
}$, satisfies the conditions in Corollary~\ref{corMultConj}.  Since $b=7$ and ${\rm gcd}(x,z) = 2$, we see that $2b$ is the smallest integral multiple of $y= 14$.  Thus, adding any positive multiple of $14$ to all the entries of ${\bf a}$ will provide
a new Gorenstein sequence which is not a complete intersection.
Observe that adding $b=7$ to all entries in ${\bf a}$ they have a common factor of $2$ and it does not
result in a Gorenstein sequence.
}\end{example}

\begin{remark}{\rm
When two rows of the principal matrix different from the second and the fourth ones both have entries summing up to zero, one does not expect to see
Gorenstein curves that are not complete intersections if we take $(a+t, a+x+t, a+y+t, a+z+t)$ for $t$ large. For example,
the sequence ${\bf a}=(43,67,49,83)$ is Gorenstein and not a complete intersection and its principal matrix,
$\displaystyle{
D({\bf a})=
\left[
\begin{matrix}
-5&0&1&2\\
2&-5&0&3\\
3&1&-4&0\\
0&4&3&-5\\
\end{matrix}
\right]
}$, satisfies that both its second and third rows have entries summing up to zero.
Adding $t(1,1,1,1)$ to ${\bf a}$ will provide a Gorenstein sequence which is not a complete intersection for $t=15$, $49$ and $83$ but
does not seem to result in a Gorenstein sequence which is not a complete intersection for larger values of $t$.
}\end{remark}

We will end this note giving a precise description of a minimal graded free resolution of $S({\bf a})$ as an $R$-module
when ${\bf a}$ is Gorenstein and not a complete intersection. Assume that ${\bf a}=(a,a+x,a+y,a+z)$ is Gorenstein but not a complete intersection
and let $D({\bf a})$ be the principal matrix of ${\bf a}$ given in (\ref{embdim4GorMat}).
Since all Gorenstein grade 3 ideals in $k[x_1,\ldots,x_n]$ must be the ideal of $n$ order pfaffians of an $(n+1)\times (n+1)$ skew symmetric matrix,
so must the ideal $I({\bf a})$.
The ideal $I ({\bf a})$ described in \cite{Br} is indeed the ideal of minors of the skew symmetric matrix
$$
\phi({\bf a}) = \left[
\begin{matrix}
0&0& x_2^{d_{32}} &x_3^{d_{43}}&x_4^{d_{24}}\\
0&0&x_1^{d_{21}}&x_4^{d_{14}}&x_2^{d_{42}}\\
-x_2^{d_{32}}&-x_1^{d_{21}}&0&0&x_3^{d_{13}}\\
-x_3^{d_{43}}&-x_4^{d_{14}}&0&0&x_1^{d_{31}}\\
-x_4^{d_{24}}&-x_2^{d_{42}}&-x_3^{d_{13}}&-x_1^{d_{31}}&0\\
\end{matrix}
\right]\,.
$$
The graded resolution of $S({\bf a})$ is
$$
0 \rightarrow R(-(ac_1+(a+z)c_4+(a+x)d_{32})) \stackrel{\delta_3}{\rightarrow} R^5 \stackrel{\phi}{\rightarrow} R^{5}\stackrel{\delta_1} \rightarrow R \rightarrow S({\bf a}) \rightarrow 0
$$
where $\phi = \phi ({\bf a})$ and $\delta _1= (\delta _3) ^T =\delta({\bf a})$ for
$$\delta({\bf a}) =
[x_1^{c_1}-x_3^{d_{13}}x_4^{d_{14}},
x_3^{c_3}-x_1^{d_{31}}x_2^{d_{32}},
x_4^{c_4}-x_2^{d_{42}}x_3^{d_{43}},
x_2^{c_2}-x_1^{d_{21}}x_4^{d_{24}}, x_1^{d_{21}}x_3^{d_{43}}-x_2^{d_{32}}x_4^{d_{14}}]\,.$$
Observe that the socle degree is ${\bf a} . [c_1, d_{32}, 0, c_4]-3$ where $.$ is the dot product of the vectors.

\medskip

\begin{remark}{\rm
By following the proof of the theorem \ref {main}, we see that when
one translates ${\bf a}$ through ${\bf u}$, we get that
$$
\phi ({\bf a}+t {\bf u})=   \left[
\begin{matrix}
0&0& x_2^{d_{32}} &x_3^{d_{43}}&x_4^{d_{24}}\\
0&0&x_1^{d_{21}}&x_4^{d_{14}}&x_2^{d_{42}}\\
-x_2^{d_{32}}&-x_1^{d_{21}}&0&0&x_3^{d_{13}+t}\\
-x_3^{d_{43}}&-x_4^{d_{14}}&0&0&x_1^{d_{31}+t}\\
-x_4^{d_{24}}&-x_2^{d_{42}}&-x_3^{d_{13}+t}&-x_1^{d_{31}+t}&0\\
\end{matrix}
\right]
$$
and the socle degree is increased by
$t^2u_1+t(u_1c_1+u_2d_{32}+u_4c_4+a)$ to get
$ t^2u_1 +t a + [{\bf a}+ t {\bf u}]. [c_1, d_{32}, 0, c_4]-3$.
If one translates through $\bf {v}$, we get
$$
\phi ({\bf a}+t{\bf v})= \left[
\begin{matrix}
0&0& x_2^{d_{32}} &x_3^{d_{43}}&x_4^{d_{24}+t}\\
0&0&x_1^{d_{21}}&x_4^{d_{14}}&x_2^{d_{42}+t}\\
-x_2^{d_{32}}&-x_1^{d_{21}}&0&0&x_3^{d_{13}}\\
-x_3^{d_{43}}&-x_4^{d_{14}}&0&0&x_1^{d_{31}}\\
-x_4^{d_{24}+t}&-x_2^{d_{42}+t}&-x_3^{d_{13}}&-x_1^{d_{31}}&0\\
\end{matrix}
\right]
$$ and the socle degree increases by
$t^2v_4+t(v_1c_1+v_2d_{32}+v_4c_4+a+z)$
to get  $ t^2v_4 +t (a+z) + [{\bf a}+ t{\bf v}]. [c_1, d_{32}, 0, c_4]-3$.
}
\end{remark}

\end{document}